\documentclass{amsart}
\usepackage{amssymb,amsmath,amsthm,amscd,mathtools,pxfonts,pdflscape,multirow,float,rotating}

\newtheorem{theorem}{Theorem}
\newtheorem{lemma}[theorem]{Lemma}
\newtheorem{corollary}[theorem]{Corollary}
\newtheorem{proposition}[theorem]{Proposition}

\newtheorem{remark}[theorem]{Remark}
\numberwithin{theorem}{section}
\numberwithin{equation}{section}

\newcommand*{\MyScale}{0.75}

\restylefloat{table}
\newcommand{\nocontentsline}[3]{}
\newcommand{\tocless}[2]{\bgroup\let\addcontentsline=\nocontentsline#1*{#2}\egroup}

\newcommand\mycom[2]{\genfrac{}{}{0pt}{}{#1}{#2}}
\newcommand{\app}[4]{F_{#1}\! \left(\left. \mycom{#2}{#3} \right| #4 \right) }
\newcommand{\appd}[4]{\mathsf{F}_{#1}\! \left(\left. \mycom{#2}{#3} \right| #4 \right) }
\newcommand{\appo}[1]{F_{#1}}
\newcommand{\appdo}[1]{\mathsf{F}_{#1}}
\newcommand{\hpg}[5]{{}_{#1}F_{#2}\! \left(\left. \mycom{#3}{#4} \right| #5 \right) }
\newcommand{\hpgd}[5]{{}_{#1}\mathsf{F}_{#2}\! \left(\left. \mycom{#3}{#4} \right| #5 \right) }
\newcommand{\hpgo}[2]{{}_{#1}F_{#2}}
\newcommand{\hpgdo}[2]{{}_{#1}\mathsf{F}_{#2}}

\newcommand{\re}{\textnormal{Re}}

\begin{document}

\title[Kummer surfaces and special function identities]{Jacobian elliptic Kummer surfaces \\[0.2em] and special function identities}

\author{Elise Griffin}
\address{Department of Mathematics and Statistics, Utah State University,
Logan, UT 84322}
\email{elise.griffin@aggiemail.usu.edu}

\author{Andreas Malmendier}
\address{Department of Mathematics and Statistics, Utah State University,
Logan, UT 84322}
\email{andreas.malmendier@usu.edu}

\begin{abstract}
We derive formulas for the construction of all inequivalent Jacobian elliptic fibrations on the Kummer surface of two non-isogeneous elliptic curves 
from extremal rational elliptic surfaces by rational base transformations and quadratic twists.
We then show that each such decomposition yields a description of the Picard-Fuchs system satisfied by the periods of the holomorphic two-form 
as either a tensor product of two Gauss' hypergeometric differential equations, an Appell hypergeometric system, or a GKZ differential system.
As the answer must be independent of the fibration used, identities relating differential systems are obtained.
They include a new identity relating Appell's hypergeometric system to a product of two Gauss' hypergeometric differential equations by a cubic 
transformation.
\end{abstract}


\maketitle

\section{Introduction}
\label{Introduction}
In \cite{MR1013073}  Oguiso studied the Kummer surface $\mathcal{Y}=\operatorname{Kum}(\mathcal{E}_1\times \mathcal{E}_2)$ obtained by the minimal resolution of the quotient surface of the product abelian surface  $\mathcal{E}_1\times \mathcal{E}_2$ by the inversion automorphism, where the elliptic curves $\mathcal{E}_i$ for $i=1,2$ are not mutually isogenous. As it is well known, such a Kummer surface $\mathcal{Y}$ is an algebraic $K3$ surface of Picard rank $18$ and can be equipped with Jacobian elliptic fibrations.
Oguiso classified them, and proved that on $\mathcal{Y}$ there are eleven distinct Jacobian elliptic fibrations, labeled $\mathcal{J}_1, \dots, \mathcal{J}_{11}$. Kuwata and Shioda furthered Oguiso's work in \cite{MR2409557} where they computed elliptic parameters and Weierstrass equations for all
eleven different fibrations, and analyzed the reducible fibers and Mordell-Weil lattices. 

These Weierstrass equations
are in fact families of minimal Jacobian elliptic fibrations over a two-dimensional moduli space.
We denote by $\lambda_i \in \mathbb{P}^1 \backslash \lbrace 0, 1, \infty \rbrace$ for $i=1, 2$ the modular parameter for the
elliptic curve $\mathcal{E}_i$ defined by the Legendre form
\begin{equation}
 y_i^2  = x_i \, (x_i-1) \, (x_i - \lambda_i) \;.
\end{equation} 
The moduli space for the fibrations $\mathcal{J}_1, \dots, \mathcal{J}_{11}$ is then given by unordered pairs
\begin{equation}
\label{ModuliSpace}
 (\tau_1, \tau_2) \in \mathcal{M} = \Big( \Gamma(2) \times \Gamma(2) \Big) \rtimes \mathbb{Z}_2 \backslash \mathbb{H} \times \mathbb{H} \;,
\end{equation}
such that $\lambda_i =\lambda(\tau_i)$ where $\lambda$ is the modular lambda function of level two 
for the genus-zero, index-six congruence subgroup $\Gamma(2) \subset \operatorname{PSL}_2(\mathbb{Z})$, and
the generator of $\mathbb{Z}_2$ acts by exchanging the two parameters. 

Base changes and quadratic twists provide powerful methods to produce new elliptic surfaces from simpler ones.
Miranda and Persson provided in \cite{MR867347} a classification of all extremal rational elliptic surfaces. 
Extremal rational Jacobian elliptic surfaces are among the simplest non-trivial elliptic surfaces.  
The first goal of this article is to construct all eleven Jacobian elliptic fibrations on the Kummer surface 
$\operatorname{Kum}(\mathcal{E}_1\times \mathcal{E}_2)$ from a small number of extremal rational elliptic surfaces  
by using only these two operations. As it will turn out, four extremal rational elliptic surfaces from the list in~\cite{MR867347} will suffice.
For each elliptic fibration $\mathcal{J}_i$ for $i=1,\dots,11$ there is a two-dimensional variety in algebraic correspondence with
$\mathcal{M}$ such that the elliptic fibration $\mathcal{J}_i$  is obtained from an extremal rational Jacobian elliptic fibration 
by base change and quadratic twisting. In this way, the modular parameters $\lambda_1$ and $\lambda_2$ of the elliptic curves $\mathcal{E}_1$ and $\mathcal{E}_2$, respectively,
determine a rational base transformation and quadratic twist for an extremal rational elliptic surface (without moduli)
that yield the Jacobian elliptic fibration on $\operatorname{Kum}(\mathcal{E}_1\times \mathcal{E}_2)$. This will be proved in Section~\ref{SEC:EllFib}.

It is easy to show that the family $q: \mathcal{Y}_{\lambda_1,\lambda_2} \to \mathcal{M}$ is in fact
a projective family of smooth connected projective varieties over $\mathbb{C}$ and $q$ is a proper and smooth morphism.
Moreover, there is a unique holomorphic two-form $\omega$ (up to scaling) on each $K3$ surface  $\mathcal{Y}_{\lambda_1,\lambda_2}$, and
differential equations can be used to express the variation in the cohomology $H^{2,0}(\mathcal{Y}_{\lambda_1,\lambda_2},\mathbb{C})$
as the moduli vary. One of the fundamental problems in Hodge theory is to determine the canonical flat connection, known as the Gauss-Manin connection. 
The connection reduces to a system of differential equations satisfied by the periods of $\omega$ called the Picard-Fuchs system  \cite[Sec. 4, 21]{MR0258824}. 
Since the second homology of a $K3$ surface has rank 22 and the Picard rank of 
$\operatorname{Kum}(\mathcal{E}_1\times \mathcal{E}_2)$ is eighteen if the the two elliptic curves are not mutually isogenous,
there must be four transcendental two-cycles which upon integration with the holomorphic two-form $\omega$ will give
four linearly independent periods. In turn, the Picard-Fuchs system for $\mathcal{Y}_{\lambda_1,\lambda_2}$ must be a system of linear partial differential equations in two variables which is holonomic of rank four.

In our situation, period integrals of the holomorphic two-form $\omega$ over transcendental two-cycles on $\mathcal{Y}_{\lambda_1,\lambda_2}$ 
can be evaluated using any of the eleven elliptic fibrations. Moreover, since we are able to relate every elliptic fibration to an extremal
rational elliptic surface by a rational base transformation and quadratic twist, period integrals reduce to simple iterated double integrals
representing so-called $\mathcal{A}$-hypergeometric functions. In Section~\ref{SEC:periods}, we will determine -- using the geometry of the 
eleven fibrations -- several different descriptions for the Picard-Fuchs system. As the answer must be independent of the fibration used, we obtain identities relating
different GKZ systems, i.e., systems of linear partial differential equations satisfied by  $\mathcal{A}$-hypergeometric functions.
All identities are then summarized in Theorem~\ref{thm:GM}; among them we recover
the linear transformation law for Appell's hypergeometric system, a famous quadratic identity due to Barnes and Bailey, and a new identity relating Appell's hypergeometric system 
to a product of two Gauss' hypergeometric differential equations by a cubic transformation.

\tocless\section{Acknowledgments}\setcounter{section}{1}
The first author acknowledges support from the Undergraduate Research and Creative Opportunities Grant Program
by the Office of Research and Graduate Studies at Utah State University. 
\bigskip

\section{Elliptic fibrations}
\label{SEC:EllFib}
A surface is called a Jacobian elliptic fibration if it is a (relatively) minimal elliptic surface $\pi: \mathcal{X} \to \mathbb{P}^1$ over $\mathbb{P}^1$
with a distinguished section $S_0$. The complete list of possible singular fibers has been given by Kodaira~\cite{MR0184257}. 
It encompasses two infinite families $(I_n, I_n^*, n \ge0)$ and six exceptional cases $(II, III, IV, II^*, III^*, IV^*)$.
To each Jacobian elliptic fibration $\pi: \mathcal{X} \to \mathbb{P}^1$ there is an associated Weierstrass model $\bar{\pi}: \bar{\mathcal{X}\,}\to \mathbb{P}^1$ 
with a corresponding distinguished section $\bar{S}_0$ obtained by contracting
all components of fibers not meeting $S_0$. $\bar{\mathcal{X}\,}$ is always singular 
with only rational double point singularities and irreducible fibers, and $\mathcal{X}$ is the minimal desingularization.
If we choose $t \in \mathbb{C}$ as a local affine coordinate on $\mathbb{P}^1$, we can write $\bar{\mathcal{X}\,}$ in the Weierstrass normal form
\begin{equation}
\label{Eq:Weierstrass}
 y^2 = 4 \, x^3 - g_2(t) \, x - g_3(t) \;,
\end{equation}
where $g_2$ and $g_3$ are polynomials in $t$ of degree four and six, or, eight and twelve 
if $\mathcal{X}$ is a rational surface or a $K3$ surface, respectively. In the following, we will use $t$ and $(x,y)$ as the affine base coordinate and coordinates of the elliptic fiber for a rational elliptic surface,
and $u$ and $(X,Y)$ for an elliptic $K3$ surface. It is of course well known
how the type of singular fibers is read off from the orders of vanishing of the functions $g_2$, $g_3$ and the discriminant $\Delta= g_2^3 - 27 \, g_3^2$ 
at the singular base values. Note that the vanishing degrees of $g_2$ and $g_3$ are always less or equal to three and five, respectively,
as otherwise the singularity of $\bar{X}$ is not a rational double point.

For a family of Jacobian elliptic surfaces $\pi: \mathcal{X} \to \mathbb{P}^1$, the two classes in N\'eron-Severi lattice $\mathrm{NS}(\mathcal{X})$ associated 
with the elliptic fiber and section  span a sub-lattice $\mathcal{H}$ isometric to the standard hyperbolic lattice $H$ with the quadratic form $Q=x_1x_2$, and we have the following decomposition 
as a direct orthogonal sum
\begin{equation*}
 \mathrm{NS}(\mathcal{X}) = \mathcal{H} \oplus \mathcal{W} \;.
\end{equation*}
The orthogonal complement $T(\mathcal{X}) = \mathrm{NS}(\mathcal{X})^{\perp} \in H^2(\mathcal{X},\mathbb{Z})\cap H^{1,1}(\mathcal{X})$ is
called the transcendental lattice and carries the induced Hodge structure.  
Moreover, an elliptic fibration  $\pi$ is called extremal if and only if the rank of the Mordell-Weil group of sections, denoted by $\operatorname{MW}(\pi)$, 
vanishes, i.e., $\operatorname{rank} \operatorname{MW}(\pi)=0$, and the associated elliptic surface has  maximal Picard rank.

\subsection{Extremal rational elliptic surfaces}
\label{ERES}
We describe the subset of the extremal rational elliptic surfaces in \cite{MR867347} that will be needed in Section~\ref{Sec:K3}.
In Table \ref{tab:3ExtRatHg}, $g_2, g_3, \Delta, J= g_2^3 / \Delta$ are the Weierstrass coefficients, discriminant, and $J$-function; 
the ramification points of $J$ and the Kodaira-types of the fibers  over the ramification points are given, as well as the sections that 
generate the Mordell-Weil group of sections.
For the rational families of Weierstrass models in Equation~(\ref{Eq:Weierstrass}) we will use $dx/y$ as the holomorphic 
one-form on each regular fiber of $\bar{\mathcal{X}\,}\!$. It is well-known (cf.~\cite{MR927661}) that the Picard-Fuchs equation is given by the Fuchsian system 
\begin{equation}
\label{FuchsianSystem}
 \frac{d}{dt} \left( \begin{array}{c} \omega_1 \\ \eta_1 \end{array} \right) = \left( \begin{array}{ccc} - \frac{1}{12} \frac{d \ln\Delta}{dt} && \frac{3\,\delta}{2\,\Delta} \\ 
 - \frac{g_2 \, \delta}{8 \, \Delta}& & \frac{1}{12} \frac{d \ln \Delta}{dt} \end{array} \right) \cdot \left( \begin{array}{c} \omega_1 \\ \eta_1 \end{array} \right)  \;,
  \end{equation}
 where $\omega_1 = \oint_{\; \Sigma_1} \frac{dx}{y}$ and $\eta_1 = \oint_{\; \Sigma_1} \frac{x \, dx}{y}$ for each one-cycle $\Sigma_1$
 and with $\delta=3 \, g_3 \, g_2' - 2\, g_2 \, g_3' $.
We have the following lemma:
\begin{lemma}
\label{Lem1}
For $t \not \in \lbrace 0,1, \infty \rbrace$ there is a smooth family of closed one-cycles $\Sigma_1=\Sigma_1(t)$ in the first homology of the elliptic curve given by Equation~(\ref{Eq:Weierstrass})
such that the period integral $\oint_{\; \Sigma_1} \frac{dx}{y}$ for the rational elliptic surfaces in Table~\ref{tab:3ExtRatHg} with $\mu \not = 0$
reduces to the following hypergeometric function holomorphic near $t=0$
\begin{equation}
\label{rk2hgf}
\omega_1= (2 \pi i) \;\hpg21{ \mu, 1-\mu}{1}{t}.
\end{equation}
The period is annihilated by the second-order, degree-one Picard-Fuchs operator
\begin{equation}
\label{L2mu}
 \mathsf{L}_2 = \theta^2 - t \, \big(\theta + \mu\big) \, \big(\theta + 1 -\mu\big) \;.
\end{equation}
For $\mu=0$ in Table~\ref{tab:3ExtRatHg}, the period holomorphic near $t=0$ is given by
\begin{equation}
\label{rk2hgfb}
\omega_1 = (2 \pi i) \;\, \hpg21{ \frac{1}{2}, \frac{1}{2}}{1}{\lambda} \;\hpgo10\!\left.\left(\frac{1}{2} \right| t \right) 
\end{equation}
and annihilated by the first-order, degree-one Picard-Fuchs operator
\begin{equation}
\label{L1}
 \mathsf{L}_1 = \theta - t \, \left(\theta + \frac{1}{2}\right)  \;.
\end{equation}
\end{lemma}
\begin{proof}
The proof was given in \cite{Doran:2015aa}.
\end{proof}

\begin{remark}
The names of the Jacobian elliptic surfaces in Table \ref{tab:3ExtRatHg} coincide with the ones used by Miranda and Persson \cite{MR867347} and
Herfurtner \cite{MR1129371}. 
\end{remark}

\begin{remark}
The definition and basic properties of the hypergeometric functions $\hpgo21$ will be given in Section~\ref{EulerIntegrals}.
\end{remark}

\begin{remark}
\label{Rem:dual_period}
For the rational elliptic surfaces in Table~\ref{tab:3ExtRatHg} with $\mu \not =0$, there is a smooth family of closed dual one-cycles $\Sigma'_1=\Sigma'_1(t)$
such that the period integral reduces to the second, linearly independent solution annihilated by the operator~(\ref{L2mu}) that has a singular point at $t=0$ and is given by
\begin{equation}
\label{rk2hgf_dual}
\omega'_1 = \oint_{\; \Sigma_1'} \frac{dx}{y} =  \dfrac{(2 \pi i)}{t^{\mu}} \; \hpg21{ \mu, \mu}{2\mu}{\frac{1}{t}} \;.
\end{equation}
\end{remark}

\subsection{$K3$ fibrations from base transformations and twists}
\label{Sec:K3}
Rational base changes provide a convenient method to produce Jacobian elliptic $K3$ surfaces from rational elliptic surfaces. The set-up is as follows: suppose we have a rational Jacobian elliptic surface
$\pi: \mathcal{X} \to C_\mathcal{X} = \mathbb{P}^1$ over the rational base curve $C_\mathcal{X} $.  To apply a base change, we need a rational ramified cover 
$\mathbb{P}^1 \to C_\mathcal{X}=\mathbb{P}^1$ of  degree $d$ mapping surjectively to $C_\mathcal{X}$. To be precise, for each $[u:1] \in \mathbb{P}^1$ we set $t= p(u)/u^n$ for $n \in \mathbb{N}$ where 
$p$ is a polynomial of degree $d > n \ge 0$ with the following three properties: (1) the  points $t=0$ and $t=1$ have $d$ pre-images each with branch numbers zero; (2)
$t=\infty$ is a branching point with corresponding ramification points $u=\infty$ with branch number $d-n-1$ and $u=0$ with branch number $n-1$ if $n \ge 1$;
(3) there are $d$ additional ramification points not coincident with $\lbrace 0, 1, \infty\rbrace$ with branch number $1$. 
The Riemann-Hurwitz formula $g-1=B/2+d \cdot (g'-1)$ is then satisfied for $g=g'=0$, $B=(d-n-1)+(n-1)+d$.
The base change is defined as the following fiber product:
\begin{equation}
\begin{array}{ccc}
   \mathcal{Y}:=\mathcal{X} \times_{C_\mathcal{X}} \mathbb{P}^1 & \longrightarrow  & \mathbb{P}^1 \\
   \downarrow && \downarrow\\
   \mathcal{X} & \longrightarrow  & C_\mathcal{X}
\end{array}
\end{equation} 
Generically, one expects $d=2$ in order to turn a rational surface into a $K3$ surface by a rational base change. However, the extremal rational elliptic surfaces from
Section~\ref{ERES} have star-fibers at $t=\infty$ whence values with $2 \le d \le 4$ can all produce $K3$ surfaces as well. We then obtain Jacobian elliptic $K3$ surfaces
with $d$ singular fibers of the same Kodaira-type as the rational elliptic surface $\mathcal{X}$ has at $u=0$ and $u=1$, respectively.
The effect of a base change on the singular fiber at $t=\infty$  depends on the local ramification of the cover $\mathbb{P}^1 \to C_\mathcal{X}=\mathbb{P}^1$.

Two elliptic surfaces with the same $J$-map have the same singular fibers up to some quadratic twist.
The effect of a quadratic twist on the singular fibers is as follows:
\begin{equation}
I_n \leftrightarrow I_n^*, \qquad II \leftrightarrow IV^*, \qquad III \leftrightarrow III^*, \qquad IV \leftrightarrow II^* \;.
\end{equation}
It is well-known that any two elliptic surfaces that are quadratic twists of each other become isomorphic after a suitable finite base change~\cite{MR2732092}.
For us,  quadratic twisting is understood by starting with the Weierstrass equation~(\ref{Eq:Weierstrass}) and replacing it by the following Weierstrass equation
for $\bar{\mathcal{Y}\,}$
\begin{equation}
\label{Eq:Weierstrass_b}
 Y^2 = 4 \, X^3 - g_2\left(\frac{p(u)}{u^n}\right) \, T(u)^2 \; X - g_3\left(\frac{p(u)}{u^n}\right)  \, T(u)^3 \;,
\end{equation}
where $T$ is a quadratic polynomial in $u$, and we have already combined the twisting with the aforementioned rational base transformation.
We will always require that Equation~(\ref{Eq:Weierstrass_b}) is a minimal Weierstrass fibration.

We then have the following result constructing each Jacobian elliptic fibration on $\mathcal{Y}=\operatorname{Kum}(\mathcal{E}_1\times \mathcal{E}_2)$
from extremal rational elliptic surfaces:
\begin{proposition}
\label{Prop1}
We have the following statements:
\begin{enumerate}
\item[(1)] The Jacobian elliptic fibrations $\mathcal{J}_1, \dots,$ $\mathcal{J}_7, \mathcal{J}_9$ given in \cite{MR2409557,MR1013073} 
on the Kummer surface $\operatorname{Kum}(\mathcal{E}_1\times \mathcal{E}_2)$ are obtained in Equation~(\ref{Eq:Weierstrass_b}) from the extremal Jacobian elliptic surfaces given 
in Table~\ref{tab:3ExtRatHg}  by using the rational base transformations $t=t_i(u)$ and quadratic twists $T=T_i(u)$  in Table~\ref{tab:KummerFibs} for $i=1, \dots, 9$.
\item[(2)] For $i \in \lbrace 1,2,3,7,9 \rbrace$ the formulas  in Table~\ref{tab:KummerFibs} are given over the quadratic field extension $K[d_i]$ of the field $K=\mathbb{C}(\lambda_1,\lambda_2)$ of moduli 
of $\mathcal{E}_1$ and $\mathcal{E}_2$. Table~\ref{tab:KummerFibs} presents $d_i^2$ as a polynomial in terms of their elliptic modular parameters $\lambda_1$ and $\lambda_2$.
\end{enumerate}
\end{proposition}
\begin{proof}
For each fibration we apply a transformation $(Y,X) \mapsto (Y/2,X+p(\lambda_1,\lambda_2; u))$ to the elliptic fibrations in \cite{MR2409557} 
-- where $p(\lambda_1,\lambda_2; u)$ is a polynomial in the modular parameters and the affine coordinate $u$ --  to obtain a Jacobian elliptic fibration in Weierstrass normal form.
In addition, for $\mathcal{J}_5$ we apply the transformation $(Y,X,u) \mapsto (Y/u^6,X/u^4,1+1/u)$ to move the singular fibers into convenient positions.
The proof then follows by comparing the obtained  Weierstrass normal forms with the ones obtained in Equation~(\ref{Eq:Weierstrass_b}) from the extremal Jacobian elliptic surfaces given 
in Table~\ref{tab:3ExtRatHg}  by using the rational base transformations $t=t_i(u)$ and quadratic twists $T=T_i(u)$  in Table~\ref{tab:KummerFibs} for $i=1, \dots, 9$.
\end{proof}
\begin{remark}
For $\mathcal{J}_4, \mathcal{J}_5, \mathcal{J}_6$ the base transformations and twists do not depend on a quadratic field extension. 
In these cases, the decomposition into a rational base transformation and quadratic twist is well-defined over the function field $K$ itself.
\end{remark}
The remaining fibrations, i.e., $\mathcal{J}_8$, $\mathcal{J}_{10}$, and $\mathcal{J}_{11}$, are found to be related to other Jacobian elliptic 
fibrations by rational transformations that leave the holomorphic two-form invariant.
We have the following proposition:
\begin{proposition}
\label{Prop2}
The Jacobian elliptic fibrations $\mathcal{J}_8, \mathcal{J}_{10}, \mathcal{J}_{11}$ given in \cite{MR2409557, MR1013073} 
on the Kummer surface $\mathcal{Y}=\operatorname{Kum}(\mathcal{E}_1\times \mathcal{E}_2)$ are obtained from the Jacobian elliptic fibrations $\mathcal{J}_7, \mathcal{J}_{9}$, and $\mathcal{J}_{7}$,
respectively, by the rational transformations given in Table~\ref{tab:KummerRels} that leave the holomorphic two-form invariant.
\end{proposition}
\begin{proof}
The proof follows by explicit computation. The transformation $(Y,X) \mapsto (Y/2,X+p(\lambda_1,\lambda_2; u))$ 
rescales the holomorphic two-form  $\omega=du \wedge dx/y$ by a constant factor of two for \emph{each} fibration, and therefore does not affect the result.
\end{proof}

\section{Period integrals}
\label{SEC:periods}
In \cite{MR902936,MR948812} Gel'fand, Kapranov and Zelevinsky defined a general class of hypergeometric functions, encompassing 
the classical one-variable hypergeometric functions,  the Appell and Lauricella functions. Today they are known as 
GKZ hypergeometric functions and provide an elegant basis for a theory of hypergeometric functions in several variables. 
Integral representations for these functions generalizing  the classical integral transform for Gauss' hypergeometric function found by Euler 
are known as $\mathcal{A}$-hypergeometric functions and were studied in \cite{MR1080980}.

\subsection{Euler integrals}
\label{EulerIntegrals}
The classical Euler integral transform
for Gauss' hypergeometric function $\hpgo21$ for $\re(\gamma)>\re(\beta)>0$ is given by
\begin{equation}
\label{GaussIntegral}
\hpg21{\alpha,\,\beta}{\gamma}{z} = \frac{\Gamma(\gamma)}{\Gamma(\beta) \, \Gamma(\gamma-\beta)} \, \int_0^1 (1-x)^{\gamma-\beta-1} \, (1- z\, x)^{-\alpha} \; x^{\beta-1} \, dx\;.
\end{equation}
The differential equation satisfied by $\hpgo21$ is
\begin{equation} \label{eq:euler}
z(1-z)\;\frac{d^2f}{dz^2}+
\big(\gamma-(\alpha+\beta+1)\, z\big) \; \frac{df}{dz}-\alpha\,\beta\,f=0.
\end{equation}
Equation~(\ref{eq:euler}) is a Fuchsian\footnote{Fuchsian means linear homogeneous and with regular singularities.} equation with three regular singularities at $z=0$, $z=1$ and $z=\infty$ with local exponent differences equal to $1-\gamma$, $\gamma-\alpha-\beta$, and $\alpha-\beta$, respectively.
For $\alpha=1-\beta=\mu$ and $\gamma=1$, it coincides with the differential operator $\mathsf{L}_2$ in Equation~(\ref{L2mu}).
The linear differential equation satisfied by the hypergeometric function $\hpgo21$
when written as a first-order Pfaffian system will be denoted by $\hpgdo21$ with
\begin{equation}
\label{PfaffianSystem2F1}
\hpgd21{\alpha,\,\beta}{\gamma}{z}: \quad d \vec{f}_{z} = \Omega^{(\,_2F_1)}_{z} \cdot \vec{f}_{z} 
\end{equation}
for the vector-valued function
$$
  \vec{f}_{z} = \langle f(z), \,  \theta_{z} f(z) \rangle^t 
$$ 
with $\theta_{z}= z \, \partial_{z}$. The Pfaffian matrix associated with the differential equation~(\ref{eq:euler}) is given by
 \begin{equation}
 \label{connection2F1}
\Omega^{(\,_2F_1)}_{z} = \left(
\begin {array}{cc} 
 0& \frac {1}{z} \\ 
 - \frac{\alpha \beta}{z-1} 
 	& \left( \frac{1-\gamma}{z} + \frac{\gamma-\alpha-\beta-1}{z-1} \right) 
\end {array}
   \right) \; dz\;.
 \end{equation}
The outer tensor product of two rank-two Pfaffian systems is constructed by introducing $\vec{H}_{z_1,z_2} =  \vec{f}_{z_1} \boxtimes \vec{f}_{z_2} $, i.e.,
\begin{equation*}
\begin{split}
  \vec{H}_{z_1,z_2} = \, \langle f(z_1) \, f(z_2), \;  \theta_{z_1} f(z_1) \, f(z_2),  \;  f(z_1) \, \theta_{z_2} f(z_2),  \; \theta_{z_1} f(z_1) \, \theta_{z_2}f(z_2)\rangle^t \;.
\end{split}
\end{equation*}
 The associated Pfaffian system is the rank-four system
\begin{equation}
\label{PfaffianSystemSqr}
 \hpgd21{\alpha_1,\,\beta_1}{\gamma_1}{z_1} \boxtimes \hpgd21{\alpha_1,\,\beta_1}{\gamma_2}{z_2}: \quad d\vec{H}_{z_1,z_2} = \Omega^{(\,_2F_1 \boxtimes \,_2F_1)}_{z_1,z_2}\cdot \vec{H}_{z_1,z_2}
\end{equation}
with the connection form
\begin{equation}
  \Omega^{(\,_2F_1 \boxtimes \,_2F_1)}_{z_1,z_2} =  \Omega^{(\,_2F_1)}_{z_1} \boxtimes \mathbb{I} +   \mathbb{I}  \boxtimes   \Omega^{(\,_2F_1)}_{z_2}  \;.
\end{equation}

The multivariate Appell's hypergeometric function $\appo2$ has an integral representation for $\re{(\gamma_1)} > \re{(\beta_1)} > 0$ and $\re{(\gamma_2)} > \re{(\beta_2)} > 0$
given by
\begin{equation}
\label{IntegralFormula}
\begin{split}
\app2{\alpha;\;\beta_1,\beta_2}{\gamma_1,\gamma_2}{z_1, z_2} = \frac{\Gamma(\gamma_1) \, \Gamma(\gamma_2)}{\Gamma(\beta_1) \, \Gamma(\beta_2) \, \Gamma(\gamma_1 - \beta_1) \, \Gamma(\gamma_2-\beta_2)} \quad \qquad\\
\times \, \int_0^1 dt \int_0^1 dx \; 
\frac{1}{t^{1-\beta_2} \, (1-t)^{1+\beta_2-\gamma_2} \, x^{1-\beta_1} \, (1-x)^{1+\beta_1-\gamma_1} \, (1-z_1 \, x - z_2 \, t)^{\alpha}} \;.
\end{split}
\end{equation}
Appell's function $\appo2$ satisfies a Fuchsian system of partial differential equations analogous to the hypergeometric equation for the function $\hpgo21$.
The system of linear partial differential equations satisfied by $\appo2$ is given by
\begin{equation} \label{app2system}
\begin{split}
z_1(1-z_1)\frac{\partial^2F}{\partial z_1^2}-z_1z_2\frac{\partial^2F}{\partial z_1\partial z_2}
+\left(\gamma_1-(\alpha+\beta_1+1)\, z_1\right)\frac{\partial F}{\partial z_1}-\beta_1z_2\frac{\partial F}{\partial z_2}
-\alpha \beta_1F=0,\\
z_2(1-z_2)\frac{\partial^2F}{\partial z_2^2}-z_1z_2\frac{\partial^2F}{\partial z_1\partial z_2}
+\left(\gamma_2-(\alpha+\beta_2+1)\, z_2\right)\frac{\partial F}{\partial z_2}-\beta_2  z_1\frac{\partial F}{\partial z_1}
- \alpha \beta_2F=0.
\end{split}
\end{equation}
This is a holonomic system of rank four whose singular locus on $\mathbb{P}^1\times\mathbb{P}^1$ is the union of the following lines 
\begin{equation} \label{app2sing}
z_1=0,\quad z_1=1,\quad z_1=\infty, \quad z_2=0,\quad z_2=1,\quad z_2=\infty, \quad z_1+z_2=1.
\end{equation}
The system (\ref{app2system}) of differential equations satisfied by the Appell hypergeometric function 
when written as the Pfaffian system will be denoted by $\appdo2$ with
\begin{equation}
\label{PfaffianSystemF2}
 \appd2{\alpha;\;\beta_1,\beta_2}{\gamma_1,\gamma_2}{z_1, z_2}: \quad d\vec{F}_{z_1,z_2} = \Omega^{(F_2)}_{z_1,z_2}  \cdot \vec{F}_{z_1,z_2} 
\end{equation}
for the vector-valued function
$$
  \vec{F}_{z_1,z_2}  = \langle F, \; \theta_{z_1} F,   \; \theta_{z_2} F, \; \theta_{z_1}\theta_{z_2} F \rangle^t \;
$$  
with $\theta_{z_i}= z_i \, \partial_{z_i}$ for $i=1, 2$. The Pfaffian matrix associated with~(\ref{app2system}) has rank four and its explicit form
is found in~\cite{MR1086776}.

The connection between the hypergeometric function $\hpgo21$  and Appell's hypergeometric function $\appo2$ is given by an integral transform
that was proved in ~\cite{Clingher:2015aa}:
\begin{lemma}
\label{EulerIntegralTransform}
For $\re{(\gamma_1)} > \re{(\beta_1)} > 0$ and $\re{(\gamma_2)} > \re{(\beta_2)} > 0$, we have the following relation between
the hypergeometric function and Appell's hypergeometric function:
\begin{equation}
\label{IntegralTransform}
\begin{split}
 \frac{1}{A^{\alpha}} \;
 \app2{\alpha;\;\beta_1,\beta_2}{\gamma_1,\gamma_2}{\frac{1}{A},  1 - \frac{B}{A}} = - \frac{\Gamma(\gamma_2) \, (A-B)^{1-\gamma_2}}{ \Gamma(\beta_2)\, \Gamma(\gamma_2-\beta_2)} \quad \\
 \times \;  \int_A^B \frac{dt}{ (A-t)^{1-\beta_2} \, (t-B)^{1+\beta_2-\gamma_2}} \, \frac{1}{t^{\alpha}} \; \hpg21{\alpha,\,\beta_1}{\gamma_1}{\frac{1}{t}}  \;.
\end{split} 
\end{equation}
\end{lemma}

\subsection{Differential systems from fibrations $\mathcal{J}_4$, $\mathcal{J}_6$, $\mathcal{J}_7$, $\mathcal{J}_9$}
As a reminder, $\lambda_1$ and $\lambda_2$ are the modular parameters of the elliptic curves  $\mathcal{E}_1$ and $\mathcal{E}_2$, respectively.
For the fibration $\mathcal{J}_4$ we have the following lemma:
\begin{lemma}
\label{lem:J4}
The Picard-Fuchs system for the periods of the holomorphic two-form on the family 
$\operatorname{Kum}(\mathcal{E}_1\times \mathcal{E}_2)$ is given by
\begin{equation}
\label{PF_J4}
 \hpgd21{\frac{1}{2},\,\frac{1}{2}}{1}{\lambda_1} \boxtimes  \hpgd21{\frac{1}{2},\,\frac{1}{2}}{1}{\lambda_2} \;.
\end{equation}
\end{lemma}
\begin{proof}
For the Jacobian elliptic fibration $\mathcal{J}_4$ on $\operatorname{Kum}(\mathcal{E}_1\times \mathcal{E}_2)$
the holomorphic two-form is given by $\omega=du \wedge dX/Y$. There is a transcendental two-cycle $\Sigma_2$ such that the period integral 
reduces to the iterated integral
\begin{equation}
\oiint_{ \Sigma_2} \omega =  2 \,  \int_1^\infty \dfrac{dt_4}{\sqrt{t_4 \, (t_4-\lambda_1)}}  \, \oint_{\Sigma_1} \frac{dx}{y} \;.
\end{equation}
where we used Proposition~\ref{Prop1} to relate the double integral to an integral for the holomorphic one-form $dx/y$ on the extremal
rational elliptic surface $\mathcal{X}_{11}(\lambda_2)$ and then reduced the outer integration to an integration along the branch cut for the function
$\sqrt{\, t_4 \, (t_4-\lambda_1)}$. Using Lemma~\ref{Lem1} and Equation~(\ref{GaussIntegral}), we evaluate the period integral further to obtain
\begin{equation}
\begin{split}
\oiint_{ \Sigma_2} \omega &=  4 \pi i \int_1^\infty \dfrac{dt_4}{\sqrt{\, t_4 \, (t_4-\lambda_1)}} \;\hpgo10\!\left.\left(\frac{1}{2} \right| t_4 \right) \; \hpg21{ \frac{1}{2}, \frac{1}{2}}{1}{\lambda_2}  \\
& = 4 \pi^2 \;\, \hpg21{ \frac{1}{2}, \frac{1}{2}}{1}{\lambda_1}  \; \hpg21{ \frac{1}{2}, \frac{1}{2}}{1}{\lambda_2} \;.
\end{split}
\end{equation}
We can change the two-cycle $\Sigma_2$ to obtain a second, linearly-independent solution for each of the factors $\hpgdo21(\lambda_1)$ and $\hpgdo21(\lambda_2)$, respectively.
This proves that there are at least four linearly independent period integrals of the holomorphic two-form $\omega$ that are annihilated by the differential system in~(\ref{PF_J4}).
As the Picard rank of $\operatorname{Kum}(\mathcal{E}_1\times \mathcal{E}_2)$ is 18 if the the two elliptic curves are not mutually isogenous, the rank of the Picard-Fuchs system equals four, and the lemma follows.
\end{proof}
Next, we look at the fibration $\mathcal{J}_7$. Here, we will need to consider a quadratic field extension of the field $K=\mathbb{C}(\lambda_1,\lambda_2)$ of moduli for the pair 
$\mathcal{E}_1$ and $\mathcal{E}_2$. We have the following lemma:
\begin{lemma}
\label{lem:J7}
Over $K[d_7]$ with $d_7^2=\lambda_1\lambda_2$, the Picard-Fuchs system for the periods of the holomorphic two-form on the family 
$\operatorname{Kum}(\mathcal{E}_1\times \mathcal{E}_2)$ is given by
\begin{equation}
\label{PF_J7}
 \dfrac{1}{\sqrt{\lambda_1+\lambda_2+2 \, d_7}} \; \appd2{\frac{1}{2};\;\frac{1}{2},\frac{1}{2}}{1,1}{v_7,w_7} \;,
\end{equation}
where we have set
\begin{equation}
\label{transfo:PF_J7}
 \Big(v_7, w_7\Big) = \left( \dfrac{4 \, d_7}{\lambda_1+\lambda_2+2 \, d_7},  -\dfrac{(1-\lambda_1)(1-\lambda_2)}{\lambda_1+\lambda_2+2 \, d_7} \right) \;.
\end{equation}
Equivalently, the Picard-Fuchs system is given by
\begin{equation}
\label{PF_J7b}
 \dfrac{1}{\sqrt{1+\lambda_1\lambda_2+2 \, d_7}} \;  \appd2{\frac{1}{2};\;\frac{1}{2},\frac{1}{2}}{1,1}{\tilde{v}_7, \tilde{w}_7} \;,
\end{equation}
where we have set 
\begin{equation}
\label{transfo:PF_J7b}
 \Big(\tilde{v}_7, \tilde{w}_7\Big) = \left( \dfrac{4 \, d_7}{1+\lambda_1\lambda_2+2 \, d_7},  \dfrac{(1-\lambda_1)(1-\lambda_2)}{1+\lambda_1\lambda_2+2 \, d_7} \right) \;.
\end{equation}
\end{lemma}
\begin{proof}
Using the Jacobian elliptic fibration $\mathcal{J}_7$
and the holomorphic two-form $\omega=du \wedge dX/Y$, there is a transcendental two-cycle $\Sigma'_2$ such that the period integral 
reduces to the iterated integral
\begin{equation}
\oiint_{ \Sigma'_2} \omega =   2  \int_0^\infty \dfrac{du}{\sqrt{\,T_7(u)}}  \;  \oint_{\Sigma'_1} \frac{dx}{y} \;,
\end{equation}
where we used Proposition~\ref{Prop1} to relate the double integral to an integral for the holomorphic one-form $dx/y$ on the extremal
rational elliptic surface $\mathcal{X}_{411}$ and then reduced the outer integration to an integration along a branch cut.
Using Remark~\ref{Rem:dual_period} and Equation~(\ref{IntegralTransform}), we evaluate the period integral further to obtain
\begin{equation}
\begin{split}
\oiint_{ \Sigma'_2} \omega &=  - 2 \pi i  \int_{A_7}^{B_7} \dfrac{dt_7}{\sqrt{\, d_7 \; (t_7-A_7) \, (t_7-B_7)}} \; \dfrac{1}{\sqrt{t_7}} \; \hpg21{ \frac{1}{2}, \frac{1}{2}}{1}{\frac{1}{t_7}}  \\
& = \frac{4 \pi^2}{\sqrt{\, 4 \, d_7  A_7}} \; \app2{\frac{1}{2};\;\frac{1}{2},\frac{1}{2}}{1,1}{\frac{1}{A_7},  1 - \frac{B_7}{A_7}} \;,
\end{split}
\end{equation}
where we have set
\begin{equation}
\label{params:J7}
 \Big(A_7, B_7 \Big) = \left( \dfrac{\lambda_1+\lambda_2}{4 \, d_7}+\dfrac{1}{2}, \, \dfrac{1+\lambda_1 \lambda_2}{4 \, d_7}+\dfrac{1}{2} \right)\;.
\end{equation}
We can change the two-cycle $\Sigma'_2$ to obtain three more linearly independent solutions with different characteristic behavior at the lines in~(\ref{app2sing}).
The rest of the proof is analogous to the proof of Lemma~\ref{lem:J4}.
Equation~(\ref{PF_J7b}) and  Equation~(\ref{transfo:PF_J7b}) follow from swapping the roles of $A_7$ and $B_7$ in Equation~(\ref{params:J7}).
\end{proof}
The comparison of Lemma~\ref{lem:J4} and Lemma~\ref{lem:J7} proves that
the Appell hypergeometric system can be decomposed as an outer tensor product of two rank-two Fuchsian systems.
We have the following corollary:
\begin{corollary}
\label{LemmaTensorSystem}
We have the following equivalence of systems of linear differential equations in two variables holonomic of rank four:
\begin{equation}
\label{id:J4J7}
 \hpgd21{\frac{1}{2},\,\frac{1}{2}}{1}{\lambda_1} \boxtimes  \hpgd21{\frac{1}{2},\,\frac{1}{2}}{1}{\lambda_2} 
 = \dfrac{1}{\sqrt{\lambda_1+\lambda_2+2 \, d_7}} \; \appd2{\frac{1}{2};\;\frac{1}{2},\frac{1}{2}}{1,1}{v_7,w_7} \;.
\end{equation}
In particular, there is a gauge transformation $G=(G_{ij})_{i,j=1}^4$ with 
$$G_{11} = \sqrt{\lambda_1+\lambda_2+2 \, d_7}, \quad G_{1j}=0 \; \text{for $j=2,3,4$},$$
such that the connection forms satisfy
\begin{equation}
\label{RelationConnectionMatrices}
   \Omega^{(\,_2F_1 \boxtimes \,_2F_1)}_{\lambda_1,\lambda_2} = G^{-1} \cdot  \Omega^{(F_2)}_{v_7,w_7} \cdot G + G^{-1}\cdot dG \;.
\end{equation}
\end{corollary}
\begin{proof}
The second statement is a special case of a more general computation that was carried out in \cite{Clingher:2015aa} where the explicit form of the gauge transformation can be found as well.
\end{proof}
\begin{remark}
\label{RefCDM}
If a carefully crafted transcendental two-cycles is chosen for the period integral, one can relate not only the two differential systems, 
but also explicit solutions to both systems. One obtains a special case of an identity by Barnes and Bailey relating
Appell's hypergeometric function to a product of Gauss' hypergeometric functions. This stronger identity was proved in \cite{Clingher:2015aa}
using period integrals on superelliptic curves and generalized Kummer varieties for general rational parameters $(\alpha,\beta_1, \beta_2, \gamma_1, \gamma_2)$.
\end{remark}

\begin{corollary}
We have the following equivalence of systems of linear differential equations in two variables holonomic of rank four:
\begin{equation}
\label{id:J6J7}
 \appd2{\frac{1}{2};\;\frac{1}{2},\frac{1}{2}}{1,1}{w_7, v_7} = \dfrac{1}{\sqrt{1-w_7}} \;  \appd2{\frac{1}{2};\;\frac{1}{2},\frac{1}{2}}{1,1}{\frac{w_7}{w_7-1}, \frac{v_7}{1-w_7}}  \;.
\end{equation}
\end{corollary}
\begin{proof}
The proof follows from the identity
\begin{equation}
 \appd2{\frac{1}{2};\;\frac{1}{2},\frac{1}{2}}{1,1}{w_7, v_7} = \sqrt{\dfrac{\lambda_1+\lambda_2+2 \, d_7}{1+\lambda_1\lambda_2+2 \, d_7}} \;  \appd2{\frac{1}{2};\;\frac{1}{2},\frac{1}{2}}{1,1}{\tilde{w}_7, \, \tilde{v}_7}  
\end{equation}
obtained by comparing Equation~(\ref{PF_J7}) and Equation~(\ref{PF_J7b}) after working out the linear relation between the variables on the left and right hand side.
\end{proof}

\begin{remark}
\label{RefCDM2}
Equation~(\ref{id:J6J7}) can be extended to a relation not only between differential systems, but between explicit solutions. 
Equation~(\ref{id:J6J7}) is then the linear transformation for Appell's hypergeometric function $\appo2$.
This stronger identity was proved in \cite{Clingher:2015aa} for general rational parameters $(\alpha,\beta_1, \beta_2, \gamma_1, \gamma_2)$.
\end{remark}

Next, we look at the fibration $\mathcal{J}_6$. However, this fibration will not provide us with a new characterization of the Picard-Fuchs system. We have the following lemma:
\begin{lemma}
\label{lem:J6}
Over $K[d_6]$ with $d_6^2=\lambda_1\lambda_2$, the Picard-Fuchs system for the periods of the holomorphic two-form on the family 
$\operatorname{Kum}(\mathcal{E}_1\times \mathcal{E}_2)$ is given by
\begin{equation}
\label{PF_J6}
 \dfrac{1}{\sqrt{\lambda_1+\lambda_2+ 2 \, d_6}} \;  \appd2{\frac{1}{2};\;\frac{1}{2},\frac{1}{2}}{1,1}{v_6, w_6} \;,
\end{equation}
where we have set
\begin{equation}
 \Big(v_6, w_6\Big) =  \left( -\dfrac{(1-\lambda_1)(1-\lambda_2)}{\lambda_1+\lambda_2+ 2 \, d_6},   \dfrac{4 \,d_6}{\lambda_1+\lambda_2+ 2 \, d_6} \right) \;.
\end{equation}
In particular, Equation~(\ref{PF_J6}) coincides with Equation~(\ref{PF_J7}) up to swapping the order of variables.
\end{lemma}
\begin{proof}
Using the Jacobian elliptic fibration $\mathcal{J}_6$ on $\operatorname{Kum}(\mathcal{E}_1\times \mathcal{E}_2)$
and the holomorphic two-form $\omega=du \wedge dX/Y$, there is a transcendental two-cycle $\Sigma'_2$ such that the period integral 
reduces to the iterated integral
\begin{equation}
\begin{split}
\oiint_{ \Sigma'_2} \omega &= - 4 \pi i \int_{A_6}^{B_6} \dfrac{\sqrt{(1-\lambda_1)(\lambda_2-1)} \, dt_6}{\sqrt{\, p_2(t)}} \; \dfrac{1}{\sqrt{t_6}} \; \hpg21{ \frac{1}{2}, \frac{1}{2}}{1}{\frac{1}{t_6}}  \;,
\end{split}
\end{equation}
where the polynomial $p_2(t)$ is given by
$$
 p_2(t) = (1-\lambda_1)^2  (\lambda_2-1)^2  \, t^2 - 2 \, (1-\lambda_1) (\lambda_2-1) (\lambda_1+\lambda_2) \, t + (\lambda_2-\lambda_1)^2 \;,
$$
and its two roots $A_6, B_6$ are
\begin{equation}
\label{params:J6}
  \Big(A_6, B_6 \Big) = \left(\dfrac{\lambda_1+\lambda_2 \pm 2 \, d_6}{(1-\lambda_1)(\lambda_2-1)},\dfrac{\lambda_1+\lambda_2 \mp 2 \, d_6}{(1-\lambda_1)(\lambda_2-1)}\right)  \;.
\end{equation}
As in the proof of Lemma~\ref{lem:J7} we obtain
\begin{equation}
\begin{split}
\oiint_{ \Sigma'_2} \omega &=    \frac{4 \pi^2}{\sqrt{\, (1-\lambda_1)(\lambda_2-1) \,  A_6}} \; \app2{\frac{1}{2};\;\frac{1}{2},\frac{1}{2}}{1,1}{\frac{1}{A_6}, 1 - \frac{B_6}{A_6}} \;.
\end{split}
\end{equation}
Notice that swapping the roles of $A_6$ and $B_6$ in the transformation~(\ref{params:J6}) interchanges $\pm d_6$.
The rest of the proof is analogous to the one of Lemma~\ref{lem:J7}.
\end{proof}

\begin{remark}
There is a beautiful geometric reason why the Picard-Fuchs systems for fibrations $\mathcal{J}_6$ and $\mathcal{J}_7$ coincide which generalizes to lower Picard rank as well.
This will be subject of a forthcoming article.
\end{remark}

Next, we look at the fibration $\mathcal{J}_9$.  We have the following lemma:
\begin{lemma}
\label{lem:J9}
Over $K[d_9]$ with $d_9^2= (\lambda_1^2 - \lambda_1 +1) (\lambda_2^2 - \lambda_2 +1)$, the Picard-Fuchs system for the periods of the holomorphic two-form on the family 
$\operatorname{Kum}(\mathcal{E}_1\times \mathcal{E}_2)$ is given by
\begin{equation}
\label{PF_J9}
 \dfrac{1}{\sqrt{R_9 \pm S_9  + 4 \, d_9}} \;  \appd2{\frac{1}{2};\;\frac{1}{6},\frac{1}{2}}{\frac{1}{3},1}{v_9, w_9} \;,
\end{equation}
where we have set 
\begin{equation}
 \Big(v_9, w_9\Big) = \left( \dfrac{8 \,d_9}{R_9 \pm S_9  + 4 \, d_9},  \dfrac{\pm S_9}{R_9 \pm S_9  + 4 \, d_9} \right) \;,
\end{equation}
and
\begin{equation}
\begin{split}
 R_9 & = 27 \, \lambda_1 \lambda_2 (\lambda_1-1) (\lambda_2-1), \\
 S_9 & = (\lambda_1+1)(\lambda_1-2)(2\lambda_1-1) (\lambda_2+1)(\lambda_2-2)(2\lambda_2-1) \;.
\end{split}
\end{equation}
\end{lemma}
\begin{proof}
Using the Jacobian elliptic fibration $\mathcal{J}_9$ on $\operatorname{Kum}(\mathcal{E}_1\times \mathcal{E}_2)$
and the holomorphic two-form $\omega=du \wedge dX/Y$, there is a transcendental two-cycle $\Sigma'_2$ 
such that the period integral 
reduces to the following integral:
\begin{equation}
\oiint_{ \Sigma'_2} \omega =  2 \, \int_0^\infty \dfrac{du}{\sqrt{\,T_9(u)}}  \;  \oint_{\Sigma'_1} \frac{dx}{y} \;.
\end{equation}
where we used Proposition~\ref{Prop1} to relate the double integral to an integral for the holomorphic one-form $dx/y$ on the extremal
rational elliptic surface $\mathcal{X}_{211}$ and then reduced the outer integration to an integration along a branch cut.
Using Remark~\ref{Rem:dual_period} we evaluate the period integral further to obtain
\begin{equation}
\begin{split}
\oiint_{ \Sigma'_2} \omega &= - \frac{4\sqrt{3}\pi i}{\sqrt{2}}   \; \int_{A_6}^{B_6} \dfrac{dt_9}{\sqrt{\, d_9 \, (t_9-A_9) (t-B_9)}} \; \dfrac{1}{t_9^{1/6}} \; \hpg21{ \frac{1}{6}, \frac{1}{6}}{\frac{1}{3}}{\frac{1}{t_9}}  \;,
\end{split}
\end{equation}
where $A_9, B_9$ are given by
\begin{equation}
\label{params:J9}
\begin{split}
 A_9 & = \dfrac{(2 \lambda_1 \lambda_2 - \lambda_1 - \lambda_2 +2) (\lambda_1 \lambda_2 + \lambda_1 - 2 \lambda_2 +1) (\lambda_1 \lambda_2 - 2 \lambda_1 + \lambda_2 +1) }{4 \, d^3_9} - \dfrac{1}{2}\;, \\
 B_9 & = \dfrac{(2 \lambda_1 \lambda_2 -\lambda_1 - \lambda_2 -1 )(\lambda_1 \lambda_2 + \lambda_1 + \lambda_2 -2) (\lambda_1 \lambda_2 - 2 \lambda_1 - 2 \lambda_2 +1) }{4 \, d^3_9} -\dfrac{1}{2} \;.
\end{split}
\end{equation}
As in the proof of Lemma~\ref{lem:J7} we obtain
\begin{equation}
\begin{split}
\oiint_{ \Sigma''_2} \omega &=  \frac{4 \sqrt{3} \pi^2}{(8 \, d_9^3 \, A_9)^{1/6}} \; \app2{\frac{1}{6};\;\frac{1}{6},\frac{1}{2}}{\frac{1}{3},1}{\frac{1}{A_9}, 1 - \frac{B_9}{A_9}} \;.
\end{split}
\end{equation}
The rest of the proof is analogous to the one of Lemma~\ref{lem:J7}.
Equation~(\ref{PF_J7b}) and  Equation~(\ref{transfo:PF_J7b}) follow from swapping the roles of $A_9$ and $B_9$ in Equation~(\ref{params:J9}).
\end{proof}

\begin{remark}
All Appell hypergeometric systems considered in Lemmas~\ref{lem:J7},~\ref{lem:J6},~\ref{lem:J9} are systems of linear differential equations in two variables holonomic of rank four.
In addition they all satisfy
\begin{equation}
\label{QuadricProperty}
 \alpha= \beta_1 + \beta_2 - \frac{1}{2}, \; \gamma_1 = 2\beta_1\, \; \gamma_2 = 2\beta_2 \;,
\end{equation}
which implies the so-called \textbf{quadric property} as proved in~\cite{MR960834}. The quadric property for a holonomic differential system
states that linearly independent solutions are quadratically related.
It is obvious that the outer tensor product in Equation~(\ref{PF_J4}) satisfies this quadric property as well.
From a geometric point of view, the quadratic property stems from the existence of a polarization for the variation of Hodge structure defined by the family of Kummer surfaces.
\end{remark}

\subsection{Differential systems from fibrations $\mathcal{J}_8$, $\mathcal{J}_{10}$, $\mathcal{J}_{11}$}
Proposition~\ref{Prop2} proves that the elliptic fibrations $\mathcal{J}_8, \mathcal{J}_{10}, \mathcal{J}_{11}$ will not give rise to additional identities relating
differential systems beyond the results obtained for fibrations $\mathcal{J}_7$, $\mathcal{J}_{9}$.
In fact, the systems derived from fibrations $\mathcal{J}_8, \mathcal{J}_{11}$ and $\mathcal{J}_{10}$ coincide with the ones found in Lemma~\ref{lem:J7}
and Lemma~\ref{lem:J9}, respectively.

\subsection{A particular GKZ system}
\label{sec:GKZ}
For the remaining fibrations a reduction of the Picard-Fuchs system to an Appell hypergeometric system is in general not possible.
Instead, we will give a description of the differential systems by restricting a particular family of GKZ systems.

We start with the two subsets $\mathcal{A}_1, \mathcal{A}_2 \subset \mathbb{Z}^2$ given by
\begin{equation}
 \mathcal{A}_2 = \left\lbrace \left(\begin{array}{c} 0 \\ 1 \end{array} \right), \left(\begin{array}{c} 0 \\ 0 \end{array} \right) \right\rbrace \;,
 \quad
 \mathcal{A}_1 = \mathcal{A}_2 \cup \left\lbrace  \left(\begin{array}{c} 3 \\ 0 \end{array} \right),  \left(\begin{array}{c} 2 \\ 0 \end{array} \right), 
  \left(\begin{array}{c} 1 \\ 0 \end{array} \right),  \left(\begin{array}{r} -1 \\ 0 \end{array} \right) \right\rbrace \;.
\end{equation}
To each element $\mathbf{n} = (n_1, n_2) \in \mathbb{Z}^2$  we associate the Laurent monomial
$x^{\mathbf{n}} = x_1^{n_1} \, x_2^{n_2}$ in the two complex variables $x_1$ and $x_2$. We identify the vector space $\mathbb{C}^{\mathcal{A}_i}$ for $i=1, 2$
with the space of Laurent polynomials of the following form
\begin{equation}
 \begin{split}
   P_1 & =  v_{(1|3,0)} x^3_1 + v_{(1|2,0)} x^2_1 +  v_{(1|1,0)} x_1 + v_{(1|0,0)}  + v_{(1|-1,0)} x_1^{-1} + v_{(1|0,1)} x_2  \,,\\
   P_2 & =  v_{(2|0,1)} x_2 + v_{(2|0,0)} \,, 
 \end{split}
\end{equation}
where we have set
$$\mathbf{v} = \Big( v_{(1|3,0)},  v_{(1|2,0)}, v_{(1|1,0)}, v_{(1|0,0)}, v_{(1|-1,0)}, v_{(1|0,1)}, v_{(2|0,0)}, v_{(2|0,1)} \Big)\;,$$
and $\mathbf{P}=(P_1,P_2) \in \mathbb{C}^{\mathcal{A}_1} \times \mathbb{C}^{\mathcal{A}_2}$.
For $\vec{\alpha}=(\alpha_1,\alpha_2) \in \mathbb{Q}^2$ and  $\vec{\beta}=(\beta_1,\beta_2) \in \mathbb{Q}^2$ we study the
$\mathcal{A}$-hypergeometric integrals of the form
\begin{equation}
\label{Ahypergeom}
 \phi_{\Sigma_2}\Big(\vec{\alpha}, \vec{\beta} \; \big| \, \mathbf{v}\Big) = \oiint_{\Sigma_2}
 P_1(x_1,x_2)^{\alpha_1} \, P_2(x_2)^{\alpha_2} \, x_1^{\beta_1} \, x_2^{\beta_2} \, dx_1 \wedge dx_2 \;.
\end{equation}
The domain of integration is contained in $\mathcal{U}(\mathbf{P}) := (\mathbb{C}^*)^2 \backslash \cup_i \lbrace P_i=0\rbrace$
where we assume that the hypersurfaces $P_i=0$ for $i=1,2$ are smooth and intersect each other transversely. 
The one-dimensional local system on $\mathcal{U}(\mathbf{P})$ defined 
by the monodromy exponents $\alpha_i$ around $\lbrace P_i=0\rbrace$ and $\beta_j$ around $\lbrace x_j=0\rbrace$ for $i,j =1,2$
will be denoted by $\mathcal{L}$. Since the integrand is multivalued and can have singularities, one has to carefully explain the meaning 
of the integral in Equation~(\ref{Ahypergeom}). These technical points were all addressed in \cite[Sec.\!~2.2]{MR1080980}.
There, a suitable chain complex with homology $H_*(\mathcal{U}(\mathbf{P}), \mathcal{L})$ was defined such that $\phi_{\Sigma_2}$ depends only on the
homology class of $[\Sigma_2] \in H_2(\mathcal{U}(\mathbf{P}), \mathcal{L})$.
In this way, the $\mathcal{A}$-hypergeometric integrals in Equation~(\ref{Ahypergeom}) becomes a multivalued functions in the variables $\mathbf{v}$. 

For $v_{(2|0,1)}, v_{(2|0,0)}, v_{(1|0,1)}, v_{(1|-1,0)} \not = 0$, we have 
\begin{equation}
\label{ToriAction}
\begin{split}
 & \quad  \phi_{\Sigma_2}\Big(\vec{\alpha}, \vec{\beta} \; \big| \, \mathbf{v}\Big)  = \left( \frac{ v_{(2|00)} }{ v_{(2|01)} } \right)^{\alpha_2+\beta_2-\beta_1}  \left( \frac{ v_{(1|-10)} }{ v_{(1|01)} } \right)^{1+\beta_1}
 v_{(2|00)}^{\alpha_1} v_{(1|01)}^{\alpha_2} \\
 \times \; & \;\left. \phi_{\Sigma_2}\left( \vec{\alpha}, \vec{\beta} \; \right| \,  \kappa^4 \frac{v_{(1|30)}}{v_{(1|-10)}}, \kappa^3  \frac{v_{(1|20)}}{v_{(1|-10)}}, \kappa^2 \frac{v_{(1|10)}}{v_{(1|-10)}}, 
 \kappa \frac{v_{(1|00)}}{v_{(1|-10)}}, 
  1, 1, 1, 1 \right) ,
 \end{split}
\end{equation}
where we have set
\begin{equation}
 \kappa= \frac{v_{(2|0,1)} v_{(1|-1,0)}}{v_{(2|0,0)} v_{(1|0,1)}} \;.
\end{equation} 
We define an affine version of the $\mathcal{A}$-hypergeometric integral by setting
\begin{equation}
\label{Ahypergeom_aff}
 \varphi_{\Sigma_2}\Big(\vec{\alpha}, \vec{\beta} \; \big| \, w_4,  w_3,  w_2,  w_1\Big): =\phi_{\Sigma_2}\left(\vec{\alpha}, \vec{\beta} \; \big| \,w_4,  w_3,  w_2,  w_1, 1, 1, 1, 1\right) .
\end{equation}
We now construct the differential system satisfied by the $\mathcal{A}$-hypergeometric integrals in Equation~(\ref{Ahypergeom}).
Using the Cayley trick we combine the sets $\mathcal{A}_1$ and $\mathcal{A}_2$ into the finite set $\mathcal{A} \subset \mathbb{Z}^4$ with
$$
\mathcal{A} = \left\lbrace 
  \left(\begin{array}{c} 1 \\ 0 \\ \hline 3 \\ 0 \end{array}\right),  \left(\begin{array}{c} 1 \\ 0 \\ \hline 2 \\ 0 \end{array}\right),  \left(\begin{array}{c} 1 \\ 0 \\ \hline 1 \\ 0 \end{array}\right), 
  \left(\begin{array}{c} 1 \\ 0 \\ \hline 0 \\ 0 \end{array}\right),  \left(\begin{array}{r} 1 \\ 0 \\ \hline -1 \\ 0 \end{array}\right),  \left(\begin{array}{c} 1 \\ 0 \\ \hline 0 \\ 1 \end{array}\right), 
   \left(\begin{array}{c} 0 \\ 1 \\ \hline 0 \\ 0 \end{array}\right), \left(\begin{array}{c} 0 \\ 1 \\ \hline 0 \\ 1 \end{array}\right) \right\rbrace.
$$
As the union of $\mathcal{A}_1$ and $\mathcal{A}_2$ generates $\mathbb{Z}^2$ as an Abelian group, and each $\mathcal{A}_i$ contains zero, the set $\mathcal{A}$ generates
$\mathbb{Z}^4$. There is a group homomorphism $h: \mathbb{Z}^4 \to \mathbb{Z}$ such that $h(\vec{\rho})=1$ for every $\vec{\rho} \in \mathcal{A}$. The homomorphism $h$ is obtained by
taking the sum of the first two components of each vector.  This means that $\mathcal{A}$ lies in a three-dimensional affine hyperplane in $\mathbb{Z}^4$.
Denote by $L(\mathcal{A}) \subset \mathbb{Z}^\mathcal{A}$ the lattice of linear relations among the elements of $\mathcal{A}$, i.e., the set of integer row vectors $(a_{\vec{\rho}^{\, t}})_{\vec{\rho} \in \mathcal{A}}$ 
such that $\sum_{\vec{\rho} \in \mathcal{A}} a_{\vec{\rho}^{\, t}} \cdot \vec{\rho} = 0$. In our case, this lattice of relations is $L(\mathcal{A})\cong \mathbb{Z}^4$ and generated by
the following row vectors
$$
L(\mathcal{A})\cong\left\lbrack \begin{array}{rrrrrr|rr}
  a_{(1|3,0)}  & a_{(1|2,0)} &  a_{(1|1,0)} & a_{(1|0,0)} & a_{(1|-1,0)} & a_{(1|0,1)} & a_{(2|0,0)} & a_{(2|0,1)}  \\
 \hline
  0	& 0	& 0	& -1	& 0	& 1	& 1	& -1	\\
  0	& 0	& 1	& -2	& 1	& 0	& 0	& 0	\\
  0	& 1	& -2	& 1	& 0	& 0	& 0	& 0	\\
  1	& -2	& 1	& 0	& 0	& 0	& 0 	& 0  \\
  \end{array}\right\rbrack \;.
 $$
 It follows from the above construction that the quotient $\mathbb{Z}^8/L(\mathcal{A}) $ is $\mathbb{Z}^4$ and torsion free.
 The complex torus $(\mathbb{C}^*)^{\mathcal{A}}$ within $\mathbb{C}^\mathcal{A}$, i.e., within the space of all vectors $\mathbf{v}=(v_{\vec{\rho}^{\, t}})_{\vec{\rho} \in \mathcal{A}}$ that
 define pairs $\mathbf{P}=(P_1,P_2)$ of Laurent polynomials, contains a subtorus $\mathbb{T}^4$ such that
the quotient equals
$$
 (\mathbb{C}^*)^{\mathcal{A} }/ \mathbb{T}^4 = \operatorname{Hom}\Big(L(\mathcal{A}), \mathbb{C}^*\Big) \;.
$$
In order to obtain the natural space on which the $\mathcal{A}$-hypergeometric integrals in Equation~(\ref{Ahypergeom}) are defined, 
Gel'fand, Kapranov and Zelevinsky developed the theory of the \emph{secondary fan}, i.e., a complete fan of rational polyhedral 
cones in the real vector space $\operatorname{Hom}(L(\mathcal{A}), \mathbb{R})$. The associated toric variety in turn determines the domains 
of convergence for various series expansions of the solutions as discs about the special points coming from the maximal cones in the secondary fan.
Equation~(\ref{Ahypergeom_aff}) then gives an integral representation in an affine chart.

Using the components of the vectors in $\mathcal{A}$, we define the first-order linear differential operators
\begin{equation}
\begin{split}
\mathsf{Z}_1  =  \sum_{k=-1}^3 v_{(1|k,0)} \frac{\partial}{\partial v_{(1|k,0)} } + v_{(1|0,1)} \frac{\partial}{\partial v_{(1|0,1)} }, &\quad
\mathsf{Z}_2  = v_{(2|0,0)} \frac{\partial}{\partial v_{(2|0,0)}}  + v_{(2|0,1)} \frac{\partial}{\partial v_{(2|0,1)} }, \\
\mathsf{Z}_3 = \sum_{k=-1}^3 k \, v_{(1|k,0)} \frac{\partial}{\partial v_{(1|k,0)} }, & \quad
\mathsf{Z}_4 = v_{(1|0,1)} \frac{\partial}{\partial v_{(1|0,1)}}  + v_{(2|0,1)} \frac{\partial}{\partial v_{(2|0,1)} }   .
\end{split}
\end{equation}
Similarly, using $L(\mathcal{A})$ one defines the second-order linear differential operators
\begin{equation}
\begin{split}
\Box_1  = \frac{\partial^2}{\partial v_{(2|0,0)}\, \partial v_{(1|0,1)}}  - \frac{\partial^2}{\partial v_{(2|0,1)} \, \partial v_{(1|0,0)}} , & \quad
\Box_2  = \frac{\partial^2}{\partial v_{(1|-1,0)}\, \partial v_{(1|1,0)}}  - \frac{\partial^2}{\partial v_{(1|0,0)}^{\,2}} ,\\
\Box_3  = \frac{\partial^2}{\partial v_{(1|0,0)}\, \partial v_{(1|2,0)}}  - \frac{\partial^2}{\partial v_{(1|1,0)}^{\,2}}  , & \quad
\Box_4 =  \frac{\partial^2}{\partial v_{(1|1,0)}\, \partial v_{(1|3,0)}}  - \frac{\partial^2}{\partial v_{(1|2,0)}^{\,2}} .
\end{split}
\end{equation}

The following lemma follows by applying \cite[Thm.\!~2.7]{MR1080980}:
\begin{lemma}
\label{lem:GKZ}
Under the above assumptions the $\mathcal{A}$-hypergeometric integral in Equation~(\ref{Ahypergeom}) satisfies for every $\Sigma_2 \in H_2(\mathcal{U}(\mathbf{P}),\mathcal{L})$ the system $\Phi$
of linear partial differential equations with finite dimensional solution space given by
\begin{equation}
\Phi \Big(\vec{\alpha}, \vec{\beta} \; \big| \, \mathbf{v}\Big) : \quad \left\lbrace \quad
\begin{aligned}
 \Box_i \; \phi_{\Sigma_2}\Big(\vec{\alpha}, \vec{\beta} \; \big| \, \mathbf{v}\Big)  & = 0 \;, \\
 \mathsf{Z}_j \;\phi_{\Sigma_2}\Big(\vec{\alpha}, \vec{\beta} \; \big| \, \mathbf{v}\Big) & = \gamma_j \; \phi_{\Sigma_2}\Big(\vec{\alpha}, \vec{\beta} \; \big| \, \mathbf{v}\Big) \;
\end{aligned}\right.
\end{equation}
for $i, j=1, \dots, 4$ and $\vec{\gamma} = \langle \alpha_1, \alpha_2,  -\beta_1-1, -\beta_2-1 \rangle$. 
\end{lemma}

\subsection{Differential systems from fibrations $\mathcal{J}_1$, $\mathcal{J}_2$, $\mathcal{J}_3$, $\mathcal{J}_5$}
For the remaining fibrations we will give a description of the Picard-Fuchs system by restricting the particular GKZ system described in Section~\ref{sec:GKZ}.
We have the following lemma:
\begin{lemma}
\label{lem:JJ}
For $i=1,2,3$ over $K[d_i]$ with $d_i^2$ given in Table~\ref{tab:KummerFibs}, 
the Picard-Fuchs system for the periods of the holomorphic two-form on the family 
$\operatorname{Kum}(\mathcal{E}_1\times \mathcal{E}_2)$ is given by
\begin{equation}
\label{PF_JJ}
 \frac{1}{g_i} \;  \Phi \left.\left(\vec{\alpha}_i, \vec{\beta}_i \; \right| \, \mathbf{v}_i \right)\;,
\end{equation}
where $g_i, \vec{\alpha}_i, \vec{\beta}_i, \mathbf{v}_i$ are given in Table~\ref{tab:KummerGKZ}. 
In particular, the restrictions define systems of linear differential equations in two variables holonomic of rank four.
\end{lemma}
\begin{proof}
We first look at the fibration $\mathcal{J}_2$.
Using the Jacobian elliptic fibration $\mathcal{J}_2$ on $\operatorname{Kum}(\mathcal{E}_1\times \mathcal{E}_2)$
and the holomorphic two-form $\omega=du \wedge dX/Y$, there is a transcendental two-cycle $\Sigma'_2$ such that the period integral 
reduces to the iterated integral
\begin{equation}
\oiint_{ \Sigma'_2} \omega =   2  \int_0^\infty \dfrac{du}{\sqrt{\,T_2(u)}}  \;  \oint_{\Sigma'_1} \frac{dx}{y} \;,
\end{equation}
where we used Proposition~\ref{Prop1} to relate the double integral to an integral for the holomorphic one-form $dx/y$ on the extremal
rational elliptic surface $\mathcal{X}_{411}$ and then reduced the outer integration to an integration along a branch cut.
Using Remark~\ref{Rem:dual_period} and Equation~(\ref{IntegralTransform}), we evaluate the period integral further to obtain
\begin{equation}
\begin{split}
\oiint_{ \Sigma'_2} \omega &=  2  \int_0^\infty \dfrac{du}{\sqrt{\,T_2(u)}}\; \dfrac{1}{\sqrt{t_2}} \; \hpg21{ \frac{1}{2}, \frac{1}{2}}{1}{\frac{1}{t_2}}  \;.
\end{split}
\end{equation}
Using the integral representation of $\hpgo21$ in Equation~(\ref{GaussIntegral}) and Table~\ref{tab:KummerFibs}, it follows that
the period integrals are of $\mathcal{A}$-hypergeometric type and annihilated by the GKZ system
\begin{equation}
 \dfrac{1}{\sqrt{d_2}} \; \Phi \left.\left(\vec{\alpha}=\left\langle - \frac{1}{2}, - \frac{1}{2} \right\rangle, \vec{\beta}=\left\langle - \frac{1}{2}, - \frac{1}{2} \right\rangle\; \right| \, \mathbf{v}\right) \;,
\end{equation}
where we have set
\begin{equation}
 \mathbf{v} = \left(  \frac{1}{16 \, d_2}, 0, \frac{1}{2}, \frac{(2\lambda_1\lambda_2-\lambda_1-\lambda_2+2)}{8 \, d_2}, \frac{(\lambda_1-\lambda_2)^2}{16 \, d_2}, 1, 1, 1\right).
\end{equation} 
We use the torus action given in Equation~(\ref{ToriAction}) to normalize.  
For $\mathcal{J}_5$ we applied the transformation $(Y,X,u) \mapsto (Y/u^6,X/u^4,1+1/u)$ in the proof of Proposition~\ref{Prop1}.
The transformation changes only the sign of the holomorphic two-from. The result for fibration $\mathcal{J}_5$ then follows from the computation for $\mathcal{J}_2$.

For fibration $\mathcal{J}_1$ the coordinate transformation $u \mapsto \sqrt{u}$ allows us to show that the period integrals are of $\mathcal{A}$-hypergeometric type 
and annihilated by the GKZ system
\begin{equation}
 \dfrac{1}{\sqrt{d_1}} \; \Phi \left.\left(\vec{\alpha}=\left\langle - \frac{1}{2}, - \frac{1}{2} \right\rangle, \vec{\beta}=\left\langle - 1, - \frac{1}{2} \right\rangle\; \right| \, \mathbf{v}\right) \;,
\end{equation}
where we have set
\begin{equation}
 \mathbf{v} = \left( 0, 0, \frac{(1-\lambda_1)^2}{16 \, d_1}, \frac{1}{2} - \frac{(1+\lambda_1)(1+\lambda_2)}{8 \, d_1},  \frac{(1-\lambda_2)^2}{16 \, d_1},  1, 1, 1 \right) \;.
\end{equation} 
Again we use the torus action given in Equation~(\ref{ToriAction}) to normalize.  
The result for fibration $\mathcal{J}_3$ follows closely the computation for $\mathcal{J}_1$. 
\end{proof}

In summary, we have considered various quadratic field extensions of the field $K=\mathbb{C}(\lambda_1,\lambda_2)$ of moduli for the pair of elliptic 
curves $\mathcal{E}_1$ and $\mathcal{E}_2$ and derived all representations for the Picard-Fuchs system satisfied by the periods of the holomorphic two-form
that can be derived from the eleven Jacobian elliptic fibrations on the Kummer surface $\operatorname{Kum}(\mathcal{E}_1\times \mathcal{E}_2)$ of two non-isogeneous elliptic curves:

\begin{theorem}
\label{thm:GM}
The Picard-Fuchs system for the periods of the holomorphic two-form on the family $\operatorname{Kum}(\mathcal{E}_1\times \mathcal{E}_2)$
of Kummer surfaces for two non-isogeneous elliptic curves  $\mathcal{E}_1$ and $\mathcal{E}_2$ with modular parameters $\lambda_1$ and $\lambda_2$, respectively,
has the following equivalent representations as linear differential systems in two variables holonomic of rank four:
\begin{equation}
\label{form:J4}
 \hpgd21{\frac{1}{2},\,\frac{1}{2}}{1}{\lambda_1} \boxtimes  \hpgd21{\frac{1}{2},\,\frac{1}{2}}{1}{\lambda_2} \;.
\end{equation}
Over $K[d_7]$ with $d_7^2=\lambda_1\lambda_2$, the system~(\ref{form:J4}) is equivalent to the Appell hypergeometric system
\begin{equation}
\label{form:J7}
 \dfrac{1}{\sqrt{\lambda_1+\lambda_2+2 \, d_7}} \; \appd2{\frac{1}{2};\;\frac{1}{2},\frac{1}{2}}{1,1}{\dfrac{4 \, d_7}{\lambda_1+\lambda_2+2 \, d_7},  -\dfrac{(1-\lambda_1)(1-\lambda_2)}{\lambda_1+\lambda_2+2 \, d_7}} \;.
\end{equation}
Over $K[d_9]$ with $d_9^2= (\lambda_1^2 - \lambda_1 +1) (\lambda_2^2 - \lambda_2 +1)$, the system~(\ref{form:J4}) is equivalent to the Appell hypergeometric system
\begin{equation}
\label{form:J9}
 \dfrac{1}{\sqrt{R_9 + S_9  + 4 \, d_9}} \;  \appd2{\frac{1}{2};\;\frac{1}{6},\frac{1}{2}}{\frac{1}{3},1}{\dfrac{8 \,d_9}{R_9 + S_9  + 4 \, d_9},  \dfrac{S_9}{R_9 + S_9  + 4 \, d_9}} \;,
\end{equation}
with
\begin{equation}
\begin{split}
 R_9 & = 27 \, \lambda_1 (\lambda_1-1)  \lambda_2 (\lambda_2-1) \;, \\
 S_9 & = (\lambda_1+1)(\lambda_1-2)(2\lambda_1-1) (\lambda_2+1)(\lambda_2-2)(2\lambda_2-1) \;.
\end{split}
\end{equation}
Over $K[d_i]$ for $i=1,2,3$ with $d_i^2$ given in Table~\ref{tab:KummerFibs}, the system~(\ref{form:J4}) is equivalent to the restrictions of the GKZ system
introduced in Section~\ref{sec:GKZ} given by
\begin{equation}
\label{form:JJ}
 \frac{1}{g_i} \;  \Phi \left.\left(\vec{\alpha}_i, \vec{\beta}_i \; \right| \, \mathbf{v}_i \right)
\end{equation}
where $\vec{\alpha}_i, \vec{\beta}_i, g_i, \mathbf{v}_i$ are given in Table~\ref{tab:KummerGKZ}. 
\end{theorem}
\begin{proof}
The comparison Lemma~\ref{lem:J4} , Lemma~\ref{lem:J7}, Lemma~\ref{lem:J9}, and Lemma~\ref{lem:JJ} 
gives the desired result.
\end{proof}
In particular, the comparison of Equation~(\ref{form:J4}) and Equation~(\ref{form:J9}) proves that
the Appell hypergeometric system can be decomposed as an outer tensor product of two rank-two systems using a cubic transformation. 

\newpage

\begin{table}[ht]
\scalebox{\MyScale}{
\begin{tabular}{|c|lcl|c|c|c|c|} \hline
\# & \multicolumn{3}{|c|}{$g_2, g_3, \Delta, J$} & \multicolumn{4}{|c|}{ramification of $J$ and singular fibers} \\[0.2em]
\hline
 $\operatorname{MW}(\pi)$ &  \multicolumn{3}{|c|}{sections}  & $t$ & $J$ & $m(J)$ & fiber   \\
\hline
\hline
$\mathcal{X}_{11}(\lambda)$    	& $g_2$ 	& $=$ & $\frac{16}{3} (\lambda ^2 - \lambda +1)(t-1)^2$ 									& $1$ 	& $J(\lambda)$       & - & $I_0^* \; \; (D_4)$  \\[0.4em]
$\mu=0$					& $g_3$ 	& $=$ & $\frac{32}{27}(\lambda - 2)(\lambda +1)(2 \lambda -1)(t-1)^3$							& $\infty$	& $J(\lambda)$       & - & $I_0^* \; \; (D_4)$  \\[0.4em]
						& $\Delta$& $=$ & $16 \, \lambda^2 (\lambda-1)^2 \, (t-1)^6$    										& & & & \\[0em]
						& $J=J(\lambda)$& $=$ & $\frac{4 \,(\lambda^2-\lambda+1)^3}{27 \lambda^2 (\lambda-1)^2}$				& & & & \\[0.3em]
\cline{1-4}
$(\mathbb{Z}/2\mathbb{Z})^2$	& $(X,Y)_{1}$ & $=$ & $(-\frac{2}{3}(\lambda+1)(t-1),0)$ 											& & & & \\[0.4em]
						& $(X,Y)_{2}$ & $=$ & $(-\frac{2}{3}(\lambda-2)(t-1),0)$ 											& & & & \\[0.4em]
						& $(X,Y)_{3}$ & $=$ & $(\frac{2}{3}(2\lambda-1)(t-1),0)$ 											& & & & \\[0.4em]
\hline
$\mathcal{X}_{411}$ 	& $g_2$ 	& $=$ & $ \frac{1}{3} (64t^2-64t+4)$				& $\frac{1}{4} \left( 2 \pm \sqrt{3}\right)$ 				& $0$       & $3$ & smooth  \\[0.4em]
$\mu=\frac{1}{2}$	& $g_3$ 	& $=$ & $\frac{8}{27}(2t-1)(32t^2-32t-1)$			& $\frac{1}{8} \left( 4 \pm 3 \sqrt{2}\right), \frac{1}{2}$		& $1$       & $2$ & smooth  \\[0.4em]
				& $\Delta$& $=$ & $256 \, t \, (t-1)$    					& $0$										& $\infty$ & $1$ & $I_1$     \\[0.0em]
				& $J$	& $=$ & $\frac{(16t^2-16 t+1)^3}{108 t (t-1)}$		& $1$                                        		 				& $\infty$ & $1$ & $I_1$     \\[0.3em]
\cline{1-4}
$\mathbb{Z}/2\mathbb{Z}$	& $(X,Y)_1$ & $=$ & $(-\frac{4}{3}t+\frac{2}{3},0)$ 	& $\infty$										& $\infty$ & $4$ & $I_4^* \; (D_8)$ \\[0.4em]
\hline
$\mathcal{X}_{222}$ 	& $g_2$ 	& $=$ & $\frac{16}{3} (t^2-t+1)$				& $\frac{1}{2} \left( 1 \pm i \sqrt{3}\right)$ 				& $0$       & $3$ & smooth  \\[0.4em]
$\mu=\frac{1}{2}$		& $g_3$ 	& $=$ & $\frac{32}{27}(t-2)(t+1)(2t-1)$		& $-1, \frac{1}{2}, 2$								& $1$       & $2$ & smooth  \\[0.4em]
					& $\Delta$& $=$ & $1024 \, t^2 \, (t-1)^2$    			& $0$										& $\infty$ & $2$ & $I_2$     \\[0.0em]
					& $J$	& $=$ & $\frac{4 \,(t^2-t+1)^3}{27 t^2 (t-1)^2}$	& $1$                                        		 				& $\infty$ & $2$ & $I_2$     \\[0.3em]
\cline{1-4}
$(\mathbb{Z}/2\mathbb{Z})^2$	& $(X,Y)_1$ & $=$ & $(-\frac{2}{3}(t+1),0)$	 	& $\infty$										& $\infty$ & $2$ & $I_2^* \; (D_6)$ \\[0.4em]
						& $(X,Y)_2$ & $=$ & $(-\frac{2}{3}(t-2),0)$	 	&											&		&	  &			      \\[0.4em]
						& $(X,Y)_3$ & $=$ & $(\frac{2}{3}(2t-1),0)$	 	&											&		&	  &			      \\[0.4em]
\hline
$\mathcal{X}_{211}$		& $g_2$	& $=$ & $3$               					& $\infty$					                				& $0$       & $2$ & $II^*\; \; (E_8)$  \\[0.4em]
$\mu=\frac{1}{6}$ 		& $g_3$	& $=$ & $-1 + 2 \, t$    					& $\frac{1}{2}$    								& $1$	& $2$ & smooth  \\[0.4em]
					& $\Delta$& $=$ & $- 108 \, t \, (t-1)$    				& $0$										& $\infty$ 	& $1$ &  $I_1$    \\[0.0em]
\cline{1-1}
$\lbrace 0 \rbrace$		& $J$	& $=$ & $- \frac{1}{4 \, t \, (t-1)}$      	   	  	& $1$                                        						& $\infty$  & $1$ &  $I_1$     \\[0.4em]
\hline
\end{tabular}
\caption{Extremal rational elliptic surfaces}\label{tab:3ExtRatHg}}
\end{table}

\begin{landscape}
\begin{table}[ht]
\scalebox{\MyScale}{
\begin{tabular}{|c|c|c|c|}
\hline
\#  	& singular fibers 			& rational	& rational base transformation	\\[-1pt]
\cline{2-2}\cline{4-4}
	&  $\operatorname{MW}(\pi)$	& surface 	& quadratic twist, $d^2$ 		\\
\hline
\hline
&&&\\[-1.7em]
	$\mathcal{J}_{1}$ 	& $2 I_{8} + 8 I_{1}$		
& 	$\mathcal X_{411}$	
& $t_1=\dfrac{(1-\lambda_1)^2 \, u^4 - 2(1+\lambda_1)(1+\lambda_2) \, u^2 + (1-\lambda_2)^2}{16 \,  d_1 \, u^2}+ \dfrac{1}{2}$ \\[-3pt]
					& $\mathbb{Z}^2 \oplus \mathbb{Z}/2\mathbb{Z}$
&					& $T_1=- \,4 \, d_1 \, u^2, \quad d_1^2 = \lambda_1 \lambda_2$  \\
\hline
&&&\\[-1.7em]
	$\mathcal{J}_{2}$ 	& $I_{4} +I_{12} + 8I_{1} $ 
&  	$\mathcal X_{411}$	
& $t_2= \dfrac{u^4 + 2 \, (2\lambda_1\lambda_2-\lambda_1-\lambda_2+2) \, u^2 + (\lambda_1-\lambda_2)^2}{16 \, d_2 \, u} + \dfrac{1}{2}$ \\[-3pt]
					& $A_2^*[2] \oplus \mathbb{Z}/2\mathbb{Z}$
&					& $T_2=- \, 4 \, d_2 \, u, \quad d_2^2=-\lambda_1 \lambda_2 (1-\lambda_1)(1-\lambda_2)$ \\
\hline
&&&\\[-1.7em]
	$\mathcal{J}_{3}$ 	&  $2 IV^{*} + 8 I_{1}$
&	$\mathcal X_{211}$	
& $t_3=\dfrac{27\lambda_1^2(\lambda_1-1)^2\,u^4 + 2 (\lambda_1+1)(\lambda_1-2)(2\lambda_1-1)(\lambda_2+1)(\lambda_2-2)(2\lambda_2-1) \, u^2 + 27\lambda_2^2(\lambda_2-1)^2}{16\, d_3^3\, u^2} + \dfrac{1}{2}$  \\
					& $\big(A_2^*[2]\big)^2$
&  					& $T_3=-\frac{8}{3} \, d_3 \, u^2, \quad d_3^2=(\lambda_1^2-\lambda_1+1)(\lambda_2^2-\lambda_2+1)$\\
\hline
	$\mathcal{J}_{4}$ 	& $4I_{0}^{*}$ 
&	$X_{11}(\lambda_2)$
& $t_4=u$ \\[-4pt]
					& $\big(\mathbb{Z}/2\mathbb{Z}\big)^2$
& 					& $T_4=\frac{1}{2} \, u\,(u-\lambda_1)$\\
\hline
&&&\\[-1.7em]
	$\mathcal{J}_{5}$ 	& $I_{6}^{*} + 6 I_{2}$ 
& 	$\mathcal X_{222}$
& $t_5= \dfrac{-\lambda_1^2 (\lambda_2-1)^2 u^3 +\lambda_1(\lambda_2-1)(1+\lambda_1+\lambda_2-2\lambda_1\lambda_2)\, u^2-(1-\lambda_1\lambda_2)(\lambda_1+\lambda_2-\lambda_1\lambda_2) \, u}{\lambda_2(\lambda_1-1)}+1$\\
					& $\big(\mathbb{Z}/2\mathbb{Z}\big)^2$
&					& $T_5=-\frac{1}{2} \, \lambda_1 \lambda_2 (\lambda_1-1) (\lambda_2-1)$\\
\hline
&&&\\[-1.7em]
	$\mathcal{J}_{6}$ 	& $2 I_{2}^{*} + 4 I_{2}$ 
& 	$\mathcal X_{222}$
& $t_6 = \dfrac{\lambda_2\, u^2 + (\lambda_2- \lambda_1) \, u + \lambda_1}{(1-\lambda_1)(1-\lambda_2) \, u}$\\
					& $\big(\mathbb{Z}/2\mathbb{Z}\big)^2$
&					& $T_6=- \frac{1}{2} \, (\lambda_1-1) \, (\lambda_2-1) \, u^2$\\
\hline
&&&\\[-1.7em]
	$\mathcal{J}_{7}$ 	& $I_{4}^{*} + 2 I_{0}^{*} + 2 I_{1}$ 
&	$\mathcal X_{411}$
& $t_7 =\dfrac{(\lambda_1\lambda_2+1) \, u - \lambda_1 -\lambda_2}{4 \, d_7 \, (u-1)} + \dfrac{1}{2}$ \\
					& $\mathbb{Z}/2\mathbb{Z}$
&					& $T_7=d_7 \, u \, (u-1)^2, \quad d_7^2=\lambda_1\lambda_2$\\
\hline
&&&\\[-1.7em]
	$\mathcal{J}_{9}$	& $II^{*} + 2 I_{0}^{*} + 2 I_{1}$
&	$\mathcal X_{211}$
& $t_9 =\dfrac{B_9 \, u - A_9}{u-1}, \quad T = - \frac{2}{3} d_9 \, u \, (u-1)^2,\quad d_9^2= (\lambda_1^2 - \lambda_1 +1) (\lambda_2^2 - \lambda_2 +1)$ \\[3pt]
					& $\lbrace 0 \rbrace$
&& $A_9= \dfrac{(2 \lambda_1 \lambda_2 - \lambda_1 - \lambda_2 +2) (\lambda_1 \lambda_2 + \lambda_1 - 2 \lambda_2 +1) (\lambda_1 \lambda_2 - 2 \lambda_1 + \lambda_2 +1) }{4 \, d^3_9} - \dfrac{1}{2}$ \\[5pt]
&&& $B_9 = \dfrac{(2 \lambda_1 \lambda_2 -\lambda_1 - \lambda_2 -1 )(\lambda_1 \lambda_2 + \lambda_1 + \lambda_2 -2) (\lambda_1 \lambda_2 - 2 \lambda_1 - 2 \lambda_2 +1) }{4 \, d^3_9} -\dfrac{1}{2}$ \\[8pt]
\hline
\end{tabular}
\caption{Fibrations on $\operatorname{Kum}(\mathcal{E}_1\times \mathcal{E}_2)$ by rational base transformations and quadratic twists}\label{tab:KummerFibs}}
\end{table}
\end{landscape}

\begin{table}[ht]
\scalebox{\MyScale}{
\begin{tabular}{|c|c|c|c|}
\hline
\#  	& singular fibers 			& related	& rational transformation	\\[-1pt]
\cline{2-2}
	&  $\operatorname{MW}(\pi)$	& fibration	& 					\\
\hline
\hline
&&&\\[-1.7em]
	$\mathcal{J}_{8}$ 	& $III^* + I_2^* + 3 I_2 + I_1$	& 	$\mathcal{J}_{7}$	
							& $u_8 = \dfrac{u_7^2 \, (u_7-1)}{x_7}$ \\[5pt]
& 						&	& $x_8=\dfrac{\big(x_7- (\lambda_1-1)(\lambda_2-1) u_7^2(u_7-1)\big)\, u_7^2}{x_7^2}$  \\[5pt]
\cline{2-2}
& $\mathbb{Z}/2\mathbb{Z}$ 	& 	& $y_8 = -\dfrac{\big(x_7 - (\lambda_1-1)(\lambda_2-1)u_7^2(u_7-1)\big) \, u_7^4 \, y_7}{x_7^4}$ \\[10pt]
\hline
&&&\\[-1.7em]
	$\mathcal{J}_{10}$ 	& $I_8^*+I_0^*+4 I_1$	& 	$\mathcal{J}_{9}$	
							& $u_{10}=\dfrac{x_9}{(u_9-1)^2}$ \\[5pt]
& 						&	& $x_{10}=-\dfrac{\lambda_1 \lambda_2 (\lambda_1-1)(\lambda_2-1)\, u_9}{u_9-1}$  \\[5pt]
\cline{2-2}
& $\lbrace 0 \rbrace$		 	& 	& $y_{10} = -\dfrac{\lambda_1 \lambda_2 ( \lambda_1-1)(\lambda_2 -1) y_9}{(u_9-1)^4}$ \\[10pt]
\hline
&&&\\[-1.7em]
	$\mathcal{J}_{11}$ 	& $2 I_4^*+4 I_1$		& 	$\mathcal{J}_{7}$	
							& $u_{11} = \dfrac{\lambda_2 u_7^2 (u_7 -1)^2}{x_7}$ \\[5pt]
& 						&	& $x_{11} = -\dfrac{\lambda_2^2 (\lambda_1-1)(\lambda_2-1) \, u_7^4 \, (u_7-1)^3}{x_7^2}$  \\[5pt]
\cline{2-2}
& $\lbrace 0 \rbrace$		 	& 	& $y_{11} = \dfrac{\lambda_2^3  (\lambda_1 - 1)(\lambda_2-1) \, u_7^6 \, (u_7 - 1)^4\, y_7}{x_7^4}$ \\[10pt]
\hline
\end{tabular}
\caption{Related elliptic fibrations on $\operatorname{Kum}(\mathcal{E}_1\times \mathcal{E}_2)$}\label{tab:KummerRels}}
\end{table}

\bigskip

\begin{table}[ht]
\scalebox{\MyScale}{
\begin{tabular}{|c|c|c|c|}
\hline
\#  & $\vec{\alpha}, \vec{\beta}$ 	& $g$ & $\mathbf{v}$ \\[5pt]
\hline
\hline
	$\mathcal{J}_{1}$ 	& $\vec{\alpha}=\left\langle - \frac{1}{2}, - \frac{1}{2} \right\rangle$  &  $g_1 =\sqrt{d_1}$
& $\mathbf{v}_1=\left( 0, 0,v_{(1|1,0)}, v_{(1|0,0)}, 1, 1, 1, 1 \right)$ \\[5pt]
\cline{3-3}
					& $\vec{\beta}=\left\langle - 1, - \frac{1}{2} \right\rangle$   &
 \multicolumn{2}{|c|}{$ v_{(1|1,0)}=\frac{(1-\lambda_1)^2(1-\lambda_2)^2}{2^8 \, d_1^2}, \; v_{(1|0,0)}= \frac{1}{2} - \frac{(1+\lambda_1)(1+\lambda_2)}{8 \, d_1}$}\\[5pt]
\hline
	$\mathcal{J}_{2}$ 	& $\vec{\alpha}=\left\langle - \frac{1}{2}, - \frac{1}{2} \right\rangle$  &  $g_2 =\frac{d_2}{\lambda_1-\lambda_2}$
& $\mathbf{v}_2=\left( v_{(1|3,0)}, 0, v_{(1|1,0)}, \frac{1}{2}, 1, 1, 1, 1  \right)$ \\[5pt]
					& $\vec{\beta}=\left\langle - \frac{1}{2}, - \frac{1}{2} \right\rangle$   &
 \multicolumn{2}{|c|}{$v_{(1|3,0)} = \frac{(\lambda_1-\lambda_2)^6}{2^{16} \, d_2^4}, \; v_{(1|1,0)}= \frac{(2\lambda_1\lambda_2-\lambda_1-\lambda_2+2)(\lambda_1-\lambda_2)^2}{2^7 \, d_2^2}$}\\[5pt]
\hline
	$\mathcal{J}_{3}$ 	& $\vec{\alpha}=\left\langle - \frac{1}{6}, - \frac{5}{6} \right\rangle$  &  $g_2 =\sqrt{d_3}$
& $\mathbf{v}_3=\left( 0, 0,  v_{(1|1,0)}, v_{(1|0,0)}, 1, 1, 1, 1 \right)$ \\[5pt]
\cline{3-3}
					& $\vec{\beta}=\left\langle - 1, - \frac{5}{6} \right\rangle$   &
 \multicolumn{2}{|c|}{$ v_{(1|1,0)} =\frac{3^6\lambda_1^2\lambda_2^2(1-\lambda_1)^2(1-\lambda_2)^2}{2^8 \, d_3^6}, \; v_{(1|0,0)}= \frac{1}{2} +  \frac{(\lambda_1+1)(\lambda_1-2)(2\lambda_1-1)(\lambda_2+1)(\lambda_2-2)(2\lambda_2-1)}{8 \, d_3^3}$}\\[5pt]
\hline
	$\mathcal{J}_{5}$ 	& $\vec{\alpha}=\left\langle -\frac{1}{2}, - \frac{1}{2} \right\rangle$  &  $g_5 =\sqrt{\frac{\lambda_1(\lambda_2-1)}{\lambda_2(\lambda_1-1)}}(1-\lambda_1\lambda_2)(\lambda_1+\lambda_2-\lambda_1\lambda_2)$
& $\mathbf{v}_5=\left(  v_{(1|3,0)}, v_{(1|2,0)}, 1, 0, 0, 1, 1, 1 \right)$ \\[5pt] 
\cline{3-3}
					& $\vec{\beta}=\left\langle 0, - \frac{1}{2} \right\rangle$   & 
 \multicolumn{2}{|c|}{$v_{(1|3,0)} = - \frac{\lambda_1^2\lambda_2^2(\lambda_1-1)^2(\lambda_2-1)^2}{(1-\lambda_1\lambda_2)^3(\lambda_1+\lambda_2-\lambda_1\lambda_2)^3}, \;
 v_{(1|2,0)} =  - \frac{\lambda_1\lambda_2(\lambda_1-1)(\lambda_2-1)(1+\lambda_1+\lambda_2-2\lambda_1\lambda_2)}{(1-\lambda_1\lambda_2)^2(\lambda_1+\lambda_2-\lambda_1\lambda_2)^2}$}\\[5pt]
\hline
\end{tabular}
\caption{Restrictions of the GKZ system from Section~\ref{sec:GKZ}}\label{tab:KummerGKZ}}
\end{table}

\newpage

\bibliographystyle{amsplain}
\bibliography{ref}{} 

\end{document}